\documentclass{amsart}

\usepackage{amssymb, amsmath, amsthm, hyperref}

\usepackage{graphicx}

\usepackage[english]{babel}

\author{Rafael Torres}

\title[Generalized complex structures in dimension 4]{Constructions of generalized complex structures in dimension four}

\address{Mathematical Institute - University of Oxford\\ 24-29 St Giles\\Oxford\\OX1 3LB\\United Kingdom}

\email{torres@maths.ox.ac.uk}

\date{November 24th, 2010.}



\theoremstyle{plain}

\newtheorem{theorem}[equation]{Theorem}

\newtheorem{proposition}[equation]{Proposition}

\newtheorem{lemma}[equation]{Lemma}

\newtheorem{question}[equation]{Question}

\newtheorem{remark}[equation]{Remark}

\theoremstyle{definition}

\newtheorem{definition}[equation]{Definition}

\newtheorem{example}[equation]{Example}

\newcommand{\Z}{\mathbb{Z}}

\begin{document}

\maketitle

In this note, four-manifold theory is employed to study the existence of (twisted) generalized complex structures. It is shown that there exist (twisted) generalized complex structures that have more than one type change loci. In an example-driven fashion, (twisted) generalized complex structures are constructed on a myriad of four-manifolds, both simply and non-simply connected, which are neither complex nor symplectic.

\section{Introduction}

Twisted generalized complex structures are a generalization of complex and symplectic structures introduced by Hitchin \cite{[H]}, and developed by Gualtieri in \cite{[G1], [G]}. The existence of an almost-complex structure is the only known obstruction for the existence of a twisted generalized complex structure on a manifold so far \cite{[G]}. Given that surfaces are K\"ahler manifolds, the question of existence of such a structure becomes non-trivial first in dimension four. In \cite{[CG1], [CG]}, Cavalcanti and Gualtieri showed that generalized complex 4-manifolds form a larger set than symplectic and/or complex manifolds. They proved that a necessary and sufficient condition for the manifolds $m \mathbb{CP}^2 \# n\overline{\mathbb{CP}^2}$ to admit a generalized complex structure, is that they admit an almost-complex one.\\

The endeavor taken in this note is to study the existence of twisted generalized complex structures using 4-manifold theory. A number of non-complex and non-symplectic twisted generalized complex manifolds are produced building on recent constructions of small symplectic 4-manifolds \cite{[BKS], [FPS], [BK3], [AP]} by using the techniques of \cite{[Ta1], [KM], [Go], [ADK], [MMS], [CG1], [CG]} on Seiberg-Witten theory, symplectic and generalized complex geometry. Among the results of the paper, there are the following.\\

\begin{itemize}
\item A (twisted) generalized complex structure can have more than one type change loci.
\item The connected sums 
\begin{center}
$m(S^2\times S^2), r(S^2\times S^2) \# S^3\times S^1$,\\
$m\mathbb{CP}^2 \# n \overline{\mathbb{CP}^2}, r\mathbb{CP}^2 \# s \overline{\mathbb{CP}^2} \# S^3\times S^1$,\\
$L(p, 1)\times S^1 \# k \overline{\mathbb{CP}^2}$
\end{center}
admit a (twisted) generalized complex structure if and only if they have an almost complex structure.  In particular, the generalized complex structures on $m \mathbb{CP}^2 \# n \overline{\mathbb{CP}^2}$ built in this paper are different from the ones constructed in \cite{[CG]}. 
\item Every finitely presented fundamental group is realized by a non-symplectic twisted generalized complex 4-manifold. 
\item Constructions of twisted generalized complex 4-manifolds that do not admit a symplectic nor a complex structure, and that have specific types of fundamental groups. For example, abelian groups, free groups of arbitrary rank, and surface groups. 
\end{itemize}

The organization of the paper is the following. Section \ref{Section 2} contains a short introduction to twisted generalized complex structures. The main results used to equip four-manifolds with such a geometric structure, and a fundamental result on the study of the Seiberg-Witten invariants are stated in Section \ref{Section 3}. Generalized complex structures for spin manifolds are constructed in Section \ref{Section 4}. In Section \ref{Section 4.1}, generalized complex structures on $S^2\times S^2$ with different numbers of type change loci are constructed. The question of existence of a generalized complex structure on the connected sums $(2g - 3)(S^2\times S^2)$ is settled in Section \ref{Section 4.2}, and a preview of existence results of twisted generalized complex structures on non-simply connected manifolds is given in Section \ref{Section 4.3}. In Section \ref{Section 5}, a large class of symplectic 4-manifolds are put together in order to produce generalized complex structures that are neither complex nor symplectic. The sixth section is devoted to the study of the existence of these structures within the non-simply connected realm. In Section \ref{Section 6.1}, it is proven that all finitely presented groups are twisted generalized complex, while being neither symplectic nor complex. The last part of the paper contains non-symplectic, non-complex examples of twisted generalized complex manifolds with abelian, surface and free fundamental groups (Sections \ref{Section 6.2} and \ref{Section 6.3} respectively). The paper ends with questions for further research in Section \ref{Section 7}.

\section{Twisted generalized complex structures}\label{Section 2}

Following the work of Gualtieri in \cite{[G]}, in this section we recall the basic definitions and examples of generalized complex structures.\\

The \emph{Courant bracket} of sections of the bundle $TM \oplus T^{\ast}M$ given by the direct sum of the tangent and cotangent bundles of a smooth manifold $M$ is

\begin{center}
$[X + \xi, Y + \eta]_H := [X, Y] + \mathcal{L}_X \eta - \mathcal{L}_Y \xi - \frac{1}{2} d(\eta(X) - \xi(Y)) + i_Y i_X H$,
\end{center}
where $H$ is a closed 3-form on $M$.\\

The bundle $TM \oplus T^{\ast}M$ can be be equipped with  a natural symmetric pairing of signature $(n, n)$

\begin{center}
$< X + \xi, Y + \eta>:= \frac{1}{2} (\eta(X) + \xi(Y))$
\end{center}

as well.

\begin{definition} A \emph{twisted generalized complex structure $(M, H, \mathcal{J})$} is a complex structure $\mathcal{J}$ on the bundle $TM \oplus T^{\ast}M$ that satisfies the following two conditions.
\begin{enumerate}
\item It preserves the symmetric pairing.
\item Its $+ i$-eigenbundle, $L\subset T_{\mathbb{C}}M \oplus T_{\mathbb{C}}^{\ast}M$, is closed under the Courant bracket.
\end{enumerate}
\end{definition}

The $+i$-eigenbundle $L$ is a maximal isotropic subspace of $T_{\mathbb{C}}M \oplus T_{\mathbb{C}}^{\ast}M$ that satisfies $L \cap \overline{L} = \{ 0\}$. The bundle $L$ can be used to fully describe the complex structure $\mathcal{J}$.
Moreover, a maximal isotropic subspace $L$ may be uniquely described by a line bundle $K \subset \wedge^{\bullet}T_{\mathbb{C}}^{\ast}M$. This complex line bundle $K$ annihilated by the $+i$-eigenvalue of $\mathcal{J}$ is the \emph{canonical bundle of $\mathcal{J}$}.\\

Indeed, the twisted generalized complex structure can be completely described in terms of differential forms. In order for a complex differential form $\rho \in \wedge^{\bullet}T_{\mathbb{C}}^{\ast}M$ to be a local generator of the canonical line bundle $K$ of a twisted generalized complex structure, and thus determine $\mathcal{J}$ uniquely, it is required that $\rho$ satisfies the following three properties at every point $p \in M$.
\begin{itemize}
\item \underline{Algebraic property}: the form can be written as
\begin{center}
$\rho = e^{B + i \omega} \wedge \Omega$,
\end{center}
where $B, \omega$ are real 2-forms, and $\Omega$ is a decomposable complex form.
\item \underline{Non-degeneracy}: the non-vanishing condition
\begin{center}
$(\rho, \overline{\rho}) = \Omega \wedge \overline{\Omega} \wedge (2i\omega)^{n - k} \neq 0$
\end{center}
holds. Here $2n = dim(M)$, and $k = deg(\Omega)$.
\item \underline{Integrability}: the form $\rho$ is integrable in the sense that
\begin{center}
$d\rho + H\wedge \rho = (X + \xi) \cdot \rho$,
\end{center}
for a section $X + \xi$ of $TM \oplus T^{\ast}M$.\\
\end{itemize}

The non-degeneracy requirement is equivalent to the condition $L \cap  \overline{L} = \{0\}$. It implies that at each point of a twisted generalized complex manifold, the real subspace $ker \Omega \wedge \overline{\omega} \subset TM$ inherits a symplectic structure from the 2-form $\omega$, and a transverse complex structure is defined by the annihilator of $\omega$, as the $+i$-eigenspace of a complex structure on $T/ ker \Omega \wedge \overline{\Omega}$. 

\begin{definition} \emph{Type and parity of a twisted generalized complex structure}. Let $\rho = e^{B + i \omega} \wedge \Omega$ be a generator of the canonical bundle $K$ of a generalized complex structure $\mathcal{J}$ at a point $p\in M$. The \emph{type of $\mathcal{J}$ at $p$} is the degree of $\Omega$. The \emph{parity of $\mathcal{J}$} is the parity of its type.
\end{definition}

The type of a twisted generalized complex structure need not be constant, it may jump along a codimension two submanifold. However, the parity of the type does remain the same along connected components of the manifold $M$.

\begin{remark} {\label{Remark 1}} (Twisted) generalized complex structures: on the 3-form $H$ of our constructions. \emph{For the twisted generalized complex structures $(M, H, \mathcal{J})$ produced in this paper, the 3-form $H$ is always given by a generator of $H^3(M; \Z)$. Poincar\'e duality implies that if a 4-manifold is simply connected, then the 3-form satisfies $H = 0$. In this case,  a generalized complex structure is obtained. In particular, the non-simply connected manifolds constructed will have a twisted generalized complex structure. For example, if the fundamental group is $\pi_1(\hat{M}) \cong \Z$, then the 3-form $H$ is a generator of $H^3(M; \Z) \cong \Z$} . 
\end{remark}

Two standard examples of generalized complex manifolds are the following.

\begin{example} \emph{Complex and symplectic manifolds}.\label{Example A}
Let $(M^{2n}, I)$ be a complex manifold. Then $M^{2n}$ is a generalized complex manifold of type $n$. Indeed, we can define on $TM \oplus T^{\ast}M$, the operator

\[ \mathcal{J}_I: =\left( \begin{array}{cc}
-I & 0  \\
0 & I^{\ast}  
\end{array} \right).\] 

In this case, $T^{0, 1}M \oplus T^{\ast 1, 0}M$ is the $+i$-eigenspace of $\mathcal{J}_I$, and the canonical bundle is $K = \wedge^{n, 0}T^{\ast}M$.\\

A symplectic manifold $(M, \omega)$ is a generalized complex manifold of type 0. The operator
\[ \mathcal{J}_{\omega}: =\left( \begin{array}{cc}
0 & - \omega^{-1}  \\
\omega & 0
\end{array} \right)\]  is a complex structure on the bundle $TM \oplus T^{\ast}M$, whose $+i$-eigenspace is given by $\{X - i \omega(X): X \in T_{\mathbb{C}}M\}$; the canonical bundle $K$ is generated by the form $e^{i \omega}$.\\
\end{example}


Regarding the jump on the type of a twisted generalized complex structure, the following is observed. The projection \begin{center}$\wedge^{\bullet}T_{\mathbb{C}}^{\ast}M \rightarrow \wedge^{0}T_{\mathbb{C}}^{\ast}M$\end{center} determines a canonical section $s$ of $K^{\ast}$. In four dimensions, the vanishing of this section forces the type of a twisted generalized complex structure to jump from 0 to $2$. The jump occurs along a 2-torus, which inherits a complex structure.

\begin{definition} A point $p$ in the type-changing locus of a twisted generalized complex structure on a 4-manifold is a \emph{nondegenerate point} if it is a nondegenerate zero of $s\in C^{\infty}(K^{\ast})$.
\end{definition}

Regarding the notion of submanifolds in the generalized complex setting, we have the following.

\begin{definition} \emph{Branes}. Let $(M, H, \mathcal{J})$ be a twisted generalized complex manifold, and let $i: \Sigma \hookrightarrow M$ be a submanifold with a 2-form $F \in \Omega^2(\Sigma)$ that satisfies $dF = i^{\ast}H$, and such that the sub line bundle $\tau_F \subset (TM \oplus T^{\ast}M)|_N$ defined as
\begin{center}
$\tau_F:= \{X + \xi \in T\Sigma \oplus T^{\ast}M: i^{\ast}\xi = i_X F\}$,
\end{center}
is invariant under $\mathcal{J}$, i.e., $\tau_F$ is a complex sub-bundle. Such a submanifold $\Sigma$ is called a \emph{brane}.
\end{definition}

In the case where $M$ is a complex manifold, the definition of a brane coincides with the notion of a complex submanifold. Analogously, for symplectic manifolds, Lagrangian submanifolds provide examples of branes.

\section{Tools and techniques of construction}\label{Section 3}

The surgical methods used in the production of twisted generalized complex structures on 4-manifolds are introduced in this section.

\subsection{Torus surgeries}{\label{Section 3.1}}

Let $T$ be a 2-torus of self intersection zero inside a $4$-manifold $X$. There is a diffeomorphism $N_T \rightarrow T^2 \times D^2$ from its tubular neighborhood $N_T$ to a thick 2-torus $T^2\times D^2$. Let $\{\alpha, \beta\}$ be the generators of $\pi_1(T)$ and consider the meridian $\mu_T$ of $T$ inside $X - T$, and the push offs $S^1_{\alpha}, S^1_{\beta}$ in $\partial N_T \approx T^3$. There is no ambiguity regarding the choice of push offs, since in our constructions the manifold $X$ will be symplectic, the torus $T$ will be Lagrangian, and the push offs are taken with respect to the Lagrangian framing. The loops $S^1_{\alpha}$ and $S^1_{\beta}$ are homologous in $N_T$ to $\alpha$ and $\beta$ respectively. In particular, the set $\{S^1_{\alpha}, S^1_{\beta}, \mu\} $ forms a basis for $H_1(\partial N; \Z) \cong H_1(T^3; \Z)$.\\

The manifold obtained from $X$ by performing a $(p, q, r)$ - torus surgery on $T$ along the curve $\gamma:= p S^1_{\alpha} q S^1_{\beta}$ is defined as\\

\begin{center}
$X_{T, \gamma}(p, q, r) = (X - N_T) \cup_{\phi} (T^2 \times D^2)$,\\
\end{center}

where the diffeomorphism $\phi: T^2\times \partial D^2 \rightarrow \partial (X - N_T)$ used to glue the pieces together satisfies $\phi_*([\partial D^2]) = p[S^1_{\alpha}] + q[S^1_{\beta}] + r[\mu_T]$ in $H_1(\partial (X - N_T)); \Z)$.\\

A few words for the reader to get comfortable with our notation are in order. A $(0, 0, 1)$ -torus surgery on $T$ amounts to carving the 2-torus out, and gluing it back in exactly the same way. Therefore, if one performs  a $(0, 0, 1)$ - torus surgery along  a torus $T$ on $X$, one has $X_{T, \gamma}(0, 0, 1) = X$. Whenever $p = 0 = q$, the surgery coefficients $(p, q, r) = (0, 0, 1)$ are said to be trivial.\\

The Euler characteristic and signature of a 4-manifold are invariant under torus surgeries. If the torus $T$ is essential and the surgery coefficients are nontrivial, then \begin{center}$b_1(X_{T, \gamma}(p, q, r)) = b_1(X) - 1$ and $b_2(X_{T, \gamma}(p, q, r)) = b_2(X) - 2$.\\ \end{center}

Minding the changes on the fundamental group of the manifolds that undergo surgery, we have the following well-known result. We call the push offs $S^1_{\alpha}:= m$ and $S^1_{\beta}:= l$.

\begin{lemma}\label{Lemma BK} The fundamental group of the manifold obtained by applying a $(p, q, r)$ - torus surgery to $X$ on the torus $T$ along the curve $m + l$ is given by the quotient
\begin{center}
$\pi_1(X_{T, \gamma}(p, q, r)) \cong \pi_1(X - T)/N(\mu_T^rm_T^pl_T^q)$,
\end{center}
where $N(\mu_T^rm_T^pl_T^q)$ denotes the normal subgroup generated by $\mu_T^rm_T^pl_T^q$.
\end{lemma}

\begin{proof} cf. \cite[Proof Lemma 4]{[BKS]}. We argue in terms of the effect that the attachment of $n$-handles has on the fundamental group. The manifold $T^2\times D^2$ has a handlebody decomposition consisting of one $0$-handle, two $1$-handles, and one $2$-handle. Using the dual decomposition, in order to glue $T^2\times D^2$ back in, one attaches one $2$-handle, two $3$-handles, and one $4$-handle. The fundamental group of the resulting manifold changes as in the statement of the lemma by attaching the $2$-handle. Attaching $3$- and $4$-handles has no effect on the fundamental group of the resulting manifold.

\end{proof}


Provided that certain hypothesis on the manifold that undergoes surgery and on the corresponding torus hold, a geometric structure is readily available for $X_{T, \gamma}(p, q, r)$. If $X$ admits a symplectic form for which the torus $T$ is Lagrangian, then performing a $(p, q, 1)$-torus surgery on the preferred Lagrangian framing of $N_T$ results in $X_{T, \gamma}(p, q, 1)$ being symplectic \cite{[ADK]}; this surgical procedure is known as \emph{Luttinger surgery} \cite{[Lu], [ADK]}. The next section is devoted to describe the existence of twisted generalized complex structures on $X_{T, \gamma}(p, q, 0)$, on its blow ups and blow downs.\\

Given that the coefficients $(p, q, r)$ already encode the surgery curve $\gamma$, from now on it will be dropped from our notation: $X_T(p, q, r):= X_{T, \gamma}(p, q, r)$.

\subsection{Surgical procedures to endow manifolds with a generalized complex structure}  In \cite{[CG1]} and \cite{[CG]}, Cavalcanti and Gualtieri have studied and employed classic topological procedures in order to equip a 4-manifold with a twisted generalized complex structure.\\

Their main results that will be used for our purposes are recalled in the following three theorems.

\begin{theorem}{\label{Theorem 6}} \cite[Theorem 3.1 and Corollary 1]{[CG1]}, \cite[Theorem 4.1]{[CG]} \emph{(p, q, 0)-torus surgery}. Let $(M, \omega)$ be a symplectic 4-manifold, which contains a symplectic torus $T$ of self-intersection zero and of area $A$. Let $\hat{M}:= M_{T}(p, q, 0)$ be the result of performing a $(p, q, 0)$-torus surgery to $M$ along $T$. Then $\hat{M}$ admits a twisted generalized complex structure such that\\
\begin{itemize}

\item The complex locus is given by an elliptic curve $\Sigma$ with modular parameter $\tau = i$, and the induced holomorphic differential $\Omega$ has periods $<A, iA>$.

\item Integrability holds with respect to a 3-form $H$, which is Poincar\'e dual to $A$ times the homology class of an integral circle of $Re(\Omega^{-1})$ in $\Sigma$.\\

\end{itemize}

The $(p, q, 0)$- torus surgery can be performed simultaneously on a collection of disjoint symplectic tori in $M$.
\end{theorem}

Each time such a surgery is performed, a type change locus is obtained in the twisted generalized complex 4-manifold $M_T(p, q, 0)$. The generalized complex structures produced by Cavalcanti and Gualtieri (\cite{[CG1], [CG]}) on $m\mathbb{CP}^2 \# n \overline{\mathbb{CP}^2}$ are obtained by applying  one $(p, q, 0)$ surgery to an elliptic surface $E(m - 1)$ along a torus fiber.\\

Much like in the symplectic and complex scenarios, the existence of a generalized complex structure on a manifold is preserved under the blow up/blow down operations. The changes on the ambient manifold are exactly the same as in the complex/symplectic case. 

\begin{theorem}{\label{Theorem 7}} \cite[Theorem 3.3]{[CG1]} \emph{Blow ups}. For any non-degenerate complex point\\ $p\in M$ in a twisted generalized complex 4-manifold $M$, there exists a twisted generalized complex 4-manifold $\tilde{M}$ and a generalized complex holomorphic map $\pi: \tilde{M}\rightarrow M$ that is an isomorphism $\tilde{M} - \pi^{-1}(\{p\}) \rightarrow M - \{p\}$. The pair $(\tilde{M}, \pi)$ is called the blow up of $M$ at $p$, and it is unique up to canonical isomorphism.
\end{theorem}

\begin{theorem} {\label{Theorem 8}}\cite[Theorem 3.4]{[CG1]} \emph{Blow downs}. A twisted generalized complex\\ 4-manifold $\tilde{M}$ containing a 2-brane $\Sigma = S^2$ intersecting the complex locus in a single non-degenerate point may be blown down to a generalized complex 4-manifold $M$. That is, there is a generalized holomorphic map $\pi: \tilde{M} \rightarrow M$ to a twisted generalized complex manifold $M$ that is an isomorphism $\tilde{M} - \Sigma \rightarrow M - \{p = \pi(\Sigma)\}$.
\end{theorem}





\subsection{On the Seiberg-Witten invariants of the manifolds manufactured}  A purpose of this paper is to enlarge the class of twisted generalized complex 4-manifolds that are neither symplectic nor complex. We wish to make sure that the manifolds constructed have trivial Seiberg-Witten invariant \cite{[SW]}. From the work of Taubes \cite{[Ta1]}, such manifolds will not admit a symplectic structure.\\

The basic result for such a purpose is the adjunction inequality.

\begin{theorem} {\label{Theorem 9}} Adjunction inequality \cite{[KM], [MMS]}. Let $X$ be a smooth closed oriented 4-manifold with $b_2^+(X) > 1$, and let $\Sigma_g \hookrightarrow X$ be a smoothly embedded genus $g$ surface with non-negative self-intersection. If $g\geq 1$, then every basic class $K$ of $X$ satisfies
\begin{center}
$2g - 2 \geq |<K, [\Sigma_g]>| + [\Sigma_g] \cdot [\Sigma_g]$.
\end{center}

Furthermore, if $g = 0$ and $[\Sigma_g] \in H_2(X; \Z)$ is not a torsion class, then the Seiberg-Witten invariant of $X$ vanishes.
\end{theorem}

\section{Generalized complex structures on spin manifolds}\label{Section 4}

In this section, we occupy ourselves with equipping the almost-complex connected sums of copies of $S^2\times S^2$ with generalized complex structures. The generalized complex structures built on $S^2\times S^2$ are new, and we observe some unexpected phenomena concerning the number of type changing loci. Our constructions owe plenty to the efforts of Akhmedov, Baldridge, Fintushel, Kirk, D. Park, and Stern in \cite{[BK3], [FPS], [AP]} to construct exotic simply connected 4-manifolds with small Euler characteristics.

\subsection{More generalized complex structures on $S^2\times S^2$} \label{Section 4.1}
Perturbing the K\"ahler structure on $\mathbb{CP}^1\times \mathbb{CP}^1$ yields a generalized complex structure on $S^2\times S^2$, which has a single type change locus. In this section we to construct different generalized complex structures on $S^2\times S^2$, in the sense that they have more than one type change loci. We are indebted to Gil Cavalcanti and to Marco Gualtieri for pointing this out.\\

Consider the product of two copies of a genus 2 surface endowed with the product symplectic structure  $\Sigma_2 \times \Sigma_2$. Denote the loops generating the fundamental group  $a_i \times \{s_2\}, b_i \times \{s_2\}, \{s_1\} \times c_i, \{s_1\} \times d_i$  by $a_1, b_1, a_2, b_2$ and  $c_1, d_1, c_2, d_2$, so that $\pi_1(\Sigma_2) = <a_1, b_1, a_2, b_2 | [a_1, b_1] [a_2, b_2] =1>$ for the first factor, and analogously 
$\pi_1(\Sigma_2) = <c_1, d_1, c_2, d_2 | [c_1, d_1] [c_2, d_2] =1>$ for the second factor. The topological properties that will be used are summarized in the following proposition.\\

\begin{proposition} {\label{Proposition 10}}\emph{(cf. \cite[Section 4]{[FPS]})}. The first homology group is given by\\ $H_1(\Sigma_2\times \Sigma_2; \mathbb{Z}) \cong \mathbb{Z}^8$, and it is generated by the loops $a_1, b_1, a_2, b_2, c_1, d_1, c_2$, and $d_2$. The second homology group $H_2(\Sigma_2\times \Sigma_2; \mathbb{Z}) = \mathbb{Z}^{18}$ is generated by the sixteen tori $a_i \times c_j, a_i \times d_j, b_i\times c_j, b_i\times d_j$ $ (i = 1, 2, j = 1, 2)$, and the surfaces $\Sigma_2 \times \{ pt\}$, $\{pt\} \times \Sigma_2$. The tori $a_i\times c_j, b_i\times c_j$ and the surface $\Sigma_2 \times \{pt\}$ are geometrically dual to the tori $b_i\times d_j, a_i\times d_j$ and the surface $\{pt\} \times \Sigma_2$ respectively, in the sense that $a_i\times c_j$ and $b_i\times d_j$  intersect transversally at one point (and similarly for the other surfaces), and every other intersection is pairwise empty. The intersection form over the integers is even, and it is generated by nine hyperbolic summands.

The Euler characteristic and signature of the product manifold are given by $e(\Sigma_2\times \Sigma_2) = 4$, and $\sigma(\Sigma_2\times \Sigma_2) = 0$.\\
\end{proposition}

Notice that the element in the first homology group and the loop generating it are being denoted by the same symbol. More homotopical information on these loops will be needed in our constructions, and we will proceed as follows. Removing a surface from a 4-manifold may change the fundamental group of the ambient manifold, since non nullhomotopic loops might be introduced by the meridian of the surface, thus adding a generator to the presentation of the group. \\

According to Fintushel, D. Park and Stern \cite[Section 4]{[FPS]} inside the complement in $\Sigma_2\times \Sigma_2$ of eight of the Lagrangian tori described above, one can choose basepaths from a basepoint $(x, y)$ of $\Sigma_2\times \Sigma_2$ to a basepoint of the boundaries of the tubular neighborhoods of the tori so that the two Lagrangian push offs, and the meridians of the tori are given by\\

\begin{center}
$\{\tilde{a}_1, \tilde{c}_1; \mu_1 = [\tilde{b}_1^{-1}, \tilde{d}_1^{-1}]\}, \{\tilde{a}_1, \tilde{c}_2; \mu_2 = [\tilde{b}_1^{-1}, \tilde{d}_2^{-1}]\}, \{\tilde{a}_2, \tilde{c}_1; \mu_3 = [\tilde{b}_2^{-1}, \tilde{d}_1^{-1}]\},$
$ \{\tilde{a}_2, \tilde{c}_2; \mu_4 = [\tilde{b}_2^{-1}, \tilde{d}_2^{-1}]\}, \{\tilde{b}_1, \tilde{d}_1 \tilde{c}_1 \tilde{d}_1^{-1}; \mu_5 = [\tilde{a}_1^{-1}, \tilde{d}_1]\}, \{\tilde{b}_2, \tilde{d}_2 \tilde{c}_2 \tilde{d}_2^{-1}; \mu_6 = [\tilde{a}_2^{-1}, \tilde{d}_2]\},$
$\{\tilde{b}_2 \tilde{a}_2 \tilde{b}_2^{-1}, \tilde{d}_1; \mu_7 = [\tilde{b}_2, \tilde{c}_1^{-1}]\},  \{\tilde{b}_1 \tilde{a}_1 \tilde{b}_1^{-1}, \tilde{d}_2; \mu_8 = [\tilde{b}_1, \tilde{c}_2^{-1}]\}$.
\end{center}

The elements with tildes denote loops that are homotopic to the corresponding loop that generates an element in homology. In \cite[Section 4]{[FPS]}, homotopies among the loops are found in order to obtain a presentation for the fundamental group of complement of the tori inside $\Sigma_2 \times \Sigma_2$ that is still generated by the same number of generators that $\pi_1(\Sigma_2 \times \Sigma_2)$ has.

\begin{remark} {\label{Remark 2}} Abuse of notation. \emph{From now on we will make no distinction between a loop, its homotopy class, and the corresponding generator in homology. In particular the decorations above will be abandoned: a loop $\tilde{a}$ will now just be denoted by $a$. These oversimplifications are justified by following and building upon recent papers \cite{[BKS], [FPS], [BK3], [AP]}, where torus surgeries are used to unveil exotic smooth structures on almost-complex 4-manifolds. In particular, we are helped greatly by the analysis done by Baldridge and Kirk in \cite{[BKS], [BK3]}}.
\end{remark}

Proposition \ref{Proposition 10} says that the torus, its meridian and its Lagrangian pushoffs available for surgery are given by
\begin{center}
$T_1:= a_1\times c_1,  m_1 = a_1, l_1 = c_1, \mu_1 = [b_1^{-1}, d_1^{-1}]$,\\
$T_2 : = a_1\times c_2, m_2 = a_1, l_2 = c_2, \mu_2 = [b_1^{-1}, d_2^{-1}]$,\\
$T_3 := a_2\times c_1,  m_3 = a_2, l_3 = c_1, \mu_3 = [b_2^{-1}, d_1^{-1}]$,\\
$T_4 := a_2\times c_2, m_4 = a_2, l_4 = c_2, \mu_4 = [b_2^{-1}, d_2^{-1}]$,\\
$T_5:= b_1\times c_1, m_5 = b_1, l_5 = d_1 c_1 d_1^{-1}, \mu_5 = [a_1^{-1}, d_1]$,\\
$T_6 := b_2\times c_2, m_6 = b_2, l_6 = d_2 c_2 d_2^{-1}\mu_6 = [a_2^{-1}, d_2]$,\\
$T_7 := a_2\times d_1, m_7 = b_2 a_2 b_2^{-1}, l_7 = d_1, \mu_7 = [b_2, c_1^{-1}]$, and\\
$T_8 := a_1\times d_2, m_8 = b_1 a_1 b_1^{-1}, l_8 = d_2, \mu_8 = [b_1, c_2^{-1}]$.\\
\end{center}

The reader is kindly reminded that in our notation, for example, a $(p, 0, r)$-torus surgery on $T_1$ stands for  a torus surgery on $T_1$ along the curve $m_1^p = a_1^p$; if $r = 1$, this torus surgery is a Luttinger surgery \cite{[ADK]}, and the case $r = 0$ correspond to the surgery described in Theorem \ref{Theorem 6}.. 
We construct generalized complex structures on $S^2\times S^2$ with a prescribed number of type change loci (up to eight of them) as follows. We exemplify two cases.\\

Say we would like a generalized complex structure with a number of eight type change loci. Use Gompf's result \cite[Lemma 1.6]{[Go]} to perturb the symplectic form on $\Sigma_2\times \Sigma_2$, so that the eight homologically essential Lagrangian tori $\{T_1, \ldots, T_8\}$ become symplectic. Perform simultaneously $(1, 0, 0)$-surgeries on the tori $T_2, T_3, T_5, T_6$, and $(0, 1, 0)$-surgeries on the tori $T_1, T_4, T_7, T_8$. Let $X(8)$ be the manifold obtained after the surgeries, which is diffeomorphic to $S^2\times S^2$. By Theorem \ref{Theorem 6}, $X(8)$ admits a generalized complex structure that contains eight type change loci.\\

Let us now produce a generalized complex structure on $S^2\times S^2$ that contains four type change loci. Our starting symplectic manifold is again $\Sigma_2\times \Sigma_2$. Perturb the symplectic form so that the homologically essential Lagrangian tori $T_5, T_6, T_7$ and $T_8$ become symplectic tori \cite[Lemma 1.6]{[Go]}. Perform four Luttinger surgeries: $(0, 1, +1)$-surgery on $T_1$, $(1, 0, -1)$-surgery on $T_2$, $(1, 0,  +1)$-surgery on $T_3$, $(0, 1, -1)$-surgery on $T_4$, and $(1, 0, + 1)$-surgeries on $T_5$ and on $T_6$. Now perform four $(p, q, 0)$-torus surgeries on the remaining symplectic tori as follows. Simultaneously perform $(1, 0, 0)$-surgery on $T_5$ and on $T_6$, and $(0, 1, 0)$-surgery on $T_7$ and on $T_8$. Denote the resulting manifold by $X(4)$.\\

By Theorem \ref{Theorem 6}, $X(4)$ admits a generalized complex structure. The group $\pi_1(X)$ is generated by the elements $a_1, b_1, a_2, b_2, c_1, d_1, c_2, d_2$, and the following (among others) relations hold has the following presentation:

\begin{center}
$c_1 = [d_1^{-1}, b_1^{-1}] ,  a_1 = [b_1^{-1}, d_2^{-1}] , a_2 = [d_1^{-1}, b_2^{-1}],  c_2 = [b_2^{-1}, d_2^{-1}]$,
\end{center}
\begin{center}
$b_1 = [d_1, a_1^{-1}] = 1, b_2 = [d_2, a_2^{-1}] = 1, d_1 = [c_1^{-1}, b_2] = 1, d_2 = [c_2^{-1}, b_1] = 1$.\\
\end{center}

It is straightforward to see that $\pi_1(X(4)) = \{1\}$. Torus surgeries preserve both the Euler characteristic and the signature. Thus, $e(X(4)) = 4$ and $\sigma(X(4)) = 0$. Moreover, the generalized complex manifold $X$ is spin. Indeed, the homological effect of such a $(p, 0, 0)$ or $(0, q, 0)$-torus surgery is to kill both the homology class of the torus on which the surgery is performed, and the class of its dual torus. Therefore, the intersection form over the integers changes by a hyperbolic summand, and it remains to be even. Freedman's Theorem \cite{[F]} implies that the resulting manifold $X(4)$ is homeomorphic to $S^2\times S^2$. The four type change loci arise as the core tori of the last four torus surgeries.\\

Inside the symplectic 4-manifold obtained by applying the four Luttinger surgeries are the symplectically imbedded surfaces of genus two $\Sigma_2\times \{x\}$, and  $\{x\} \times \Sigma_2$. The $(p, 0, 0)$- and $(0, q, 0)$-torus surgeries reduce the genus of these surfaces, which after the surgeries become imbedded 2-spheres. In particular, $X(4)$ is diffeomorphic to $S^2\times S^2$.\\

Variations on the surgical procedure described above yield the following theorem

\begin{theorem}{\label{Theorem MT}} Let $n \in \{1, 2, 3, 4, 5, 6, 7, 8\}$. The 4-manifold $S^2\times S^2$ admits a generalized complex structure $(X(n), H, \mathcal{J})$, which has $n$ type change loci.
\end{theorem}

\subsection{The almost-complex connected sums $(2g - 3)(S^2\times S^2), g\geq 3$}{\label{Section 4.2}} We generalize the procedure described in the previous section with the purpose of endowing the connect sums of $S^2\times S^2$ that are almost-complex with a generalized complex structure. The first step is to consider the product $\Sigma_2\times \Sigma_g$ of a genus 2 surface with a genus $g \geq 3$ surface equipped with the product symplectic form. Let $a_i, b_i, c_j$ and $d_j$ ($i = 1, 2$, $j = 1, \ldots, g$) be the standard generators of $\pi_1(\Sigma_2)$ and $\pi_1(\Sigma_g)$ respectively. We encode the topological information we need in the following proposition.  In the notation regarding the fundamental group of the manifolds that undergo surgery, each relation is associated to the torus on which a surgery introduces the aforementioned relation.

\begin{proposition} {\label{Proposition 11}} \emph{(cf. \cite[Section 2]{[AP]})}. The first homology group is given by\\ $H_1(\Sigma_2\times \Sigma_g; \mathbb{Z}) = \mathbb{Z}^{4 + 2g}$, and it is generated by the loops $a_1, b_1, a_2, b_2, c_1, d_1, \cdots c_g, d_g$. The second homology group $H_2(\Sigma_2 \times \Sigma_g;\mathbb{Z}) = \mathbb{Z}^{8g + 2}$ is generated by $8g$ tori, and the surfaces $[\Sigma_2 \times \{pt\}], [\{pt\} \times \Sigma_g]$: the tori $a_i\times c_j$ and $a_i\times d_j$ are geometrically dual to $b_i\times d_j$ and $b_i\times c_j$ respectively, and the surface $\Sigma_2 \times \{pt\}$ is geometrically dual to $\{pt\} \times \Sigma_g$. The intersection form over the integers is even, and it is given by $4g + 1$ hyperbolic summands; the pairs of tori contribute $4g$ summands, and one summand is contributed by the surfaces $\Sigma_2\times \{pt\}$ and $\{pt\}\times \Sigma_g$. The characteristic numbers are given by $e(\Sigma_2\times \Sigma_g) = 4g - 4$, and $\sigma(\Sigma_2\times \Sigma_g) = 0$.\\

The manifold $X_g$ with $b_1(X_g) = 0$ obtained from applying $4 + 2g$ $(p, q, r)$-torus surgeries to $\Sigma_2\times \Sigma_g$ $(p, q\in \{0, 1\}, p \neq q)$ has fundamental group generated by the elements $a_1, b_1, a_2, b_2, c_1, d_1, \ldots, c_n, d_n$, and the following relations hold.
\begin{center}
$T_1:a_1 = [b_1^{-1}, d_1^{-1}]^r, T_2: b_1 = [a_1^{-1}, d_1]^r, T_3:a_2 = [b_2^{-1}, d_2^{-1}]^r, T_4: b_2  = [a_2^{-1}, d_2]^r$
\end{center}
\begin{center}
$T_5: c_1 = [d_1^{-1}, b_2^{-1}]^r, T_6: d_1  = [c_1^{-1}, b_2]^r, T_7: c_2  = [d_2^{-1}, b_1^{-1}]^r, T_8:d_2  = [c_2^{-1}, b_1]^r$
\end{center}
\begin{center}
$T_9: c_3  = [a_1^{-1}, d_3^{-1}]^r, T_{10}: d_3 = [a_2^{-1}, c_3^{-1}]^r, \ldots, T_{3 + 2g}: c_g = [a_1^{-1}, d_g^{-1}]^r$, $T_{4 + 2g}: d_g = [a_2^{-1}, c_g^{-1}]^r$.
\end{center}

\end{proposition}

The proof is a generalization of \cite[Section 4]{[FPS]}. The claims concerning the submanifolds, and their homological properties are left to the reader. A proof for the nontrivial claim regarding the fundamental group calculations can be found in \cite{[MCH]}. Notice that in the notation of Proposition \ref{Proposition 11}, the relation introduced to the fundamental group by the corresponding $(p, q, r)$-surgery appears associated to its corresponding homologically essential torus. For example, $T_1: a_1 = [b_1^{-1}, d_1^{-1}]^1$ stands for the relation coming from a $(1, 0, 1)$-surgery. The $(1, 0, 0)$-surgery of Theorem \ref{Theorem 6} applied to $T_1$ is then $T_1: a_1 =   [b_1^{-1}, d_1^{-1}]^0 = 1$.\\


We proceed to construct a generalized complex structure on the connected sums $(2g - 3)(S^2\times S^2)$ for $g\geq 3$. Using \cite[Lemma 1.6]{[Go]}, we can perturb the symplectic form on $\Sigma_2\times \Sigma_g$ so that the tori $\{T_1, T_2, T_3, \cdots T_{4 + 2g}\}$ become symplectic \cite[Lemma 1.6]{[Go]}. Let $X_g$ be the manifold of Proposition \ref{Proposition 11} obtained by applying torus surgeries with $r = 0$. Theorem \ref{Theorem 6} implies that $X_g$ admits a generalized complex structure.\\

By the presentation of the group $\pi_1(X_g)$ given in Proposition \ref{Proposition 11}, this implies that the manifold $X_g$ is simply connected. A direct computation of the characteristic numbers yields $e(X_g) = e(\Sigma_2\times \Sigma_g) = -2 \cdot (2 - 2g) = 4g - 4$, and $\sigma(X_g) = \sigma(\Sigma_2\times \Sigma_g) = 0$. By Freedman's Theorem \cite{[F]}, $X_g$ has the homeomorphism type of $(2g - 3)(S^2\times S^2)$. \\


We are ready to prove the following result.

\begin{theorem}{\label{Theorem 12}} Let $m\geq 1$. The connected sum \begin{center}$m(S^2\times S^2)$ \end{center}admits a generalized complex structure if and only if it admits an almost-complex structure.
\end{theorem}

In particular, $m = (2g - 3)$ is an odd number. The necessity of the existence of an almost-complex structure is clear \cite{[G]}. Concerning the sufficiency of the condition, it remains to be demonstrated that the manifold $X_g$ constructed above is $(2g - 3)(S^2\times S^2)$.

\begin{proof} Theorem \ref{Theorem 6} implies $X_g$ admits a generalized complex structure. We claim that $X_g$ is diffeomorphic to $(2g - 3)(S^2\times S^2)$. For the sake of clarity, we work out the case $g = 3$ by looking at the effect the surgeries have on the submanifolds of $\Sigma_2\times \Sigma_3$; the proofs for the cases $g > 3$ follow a natural generalization. The torus surgeries that were performed on $\Sigma_2\times \Sigma_3$ transform the imbedded surfaces $\Sigma_2 \times \{x\}$ and $\{x\}\times \Sigma_3$ into spheres, by reducing their genus. Moreover, according to Proposition \ref{Proposition 11}, there are four remaining homologically essential submanifolds that are geometrically dual. Inside $\Sigma_2\times \Sigma_3$, these are tori, which after the surgeries become geometrically dual 2-spheres. For example, the $(1, 0, 0)$-surgery on $T_1$ reduces the genus of the tori $a_1\times c_2, a_1\times c_3$, and the genus of their dual tori as well. Thus, becoming imbedded 2-spheres. Each of these pairs span an $S^2\times S^2$ summand in $X_g$. Therefore, $X_3$ is diffeomorphic to $3(S^2\times S^2)$. The Seiberg-Witten invariant of $X_3$ is trivial, and it is not symplectic. It is not a complex manifold by Kodaira's classification of complex surfaces \cite{[Ko], [BHPV]}.
\end{proof}

\begin{remark}{\label{Remark M}} Number of loci. \emph{Let $g\geq 3$. The generalized complex structures built in Theorem \ref{Theorem 12} have a number of $4 + 2g$ type change loci. Following the method described in Section \ref{Section 4.1}, generalized complex structures with a different number of loci can be produced.}

\end{remark}

\subsection{Preview of other fundamental groups}\label{Section 4.3} Constructions of non-simply connected 4-manifolds are given in Section \ref{Section 6}. However, At this stage we would like to point out that a variation on the coefficients of the torus surgeries in the previous constructions promptly yields twisted generalized complex 4-manifolds with several other fundamental groups. A sample for abelian groups is the following.\\

\begin{proposition} {\label{Proposition 13}} Let $g\geq 2$. The homeomorphism types corresponding to
\begin{itemize} 
\item $\pi_1 = \Z/p\Z: (2g - 3)(S^2\times S^2)\# \widetilde{L(p, 1)\times S^1}$
\item $\pi_1 = \Z/p\Z \oplus \Z/q\Z: (2g - 1)(S^2\times S^2) \# \widehat{L(p, 1)\times S^1}$
\item $\pi_1 = \Z: (2g)(S^2\times S^2)\# S^3\times S^1$
\end{itemize}
 admit a twisted generalized complex structure, which does not come from a symplectic nor a complex structure.
\end{proposition}

The pieces $\widetilde{L(p, 1)\times S^1}$ and $\widehat{L(p, 1)\times S^1}$ stand for the manifolds obtained by modifying the product $L(p, 1)\times S^1$  of a Lens space with a circle as follows. Perform a surgery on $L(p, 1)\times S^1$ along $\{x\} \times \alpha$ ($x \in L(p, 1)$) to kill the loop corresponding to the generator of the infinite cyclic group factor so that $\pi_1 = \Z/p\Z$ of the resulting manifold comes from the fundamental group of the Lens space. This amounts to removing a neighborhood of the loop $S^1\times D^3$ and glueing in a $S^2\times D^2$. Denote such a manifold by $\widetilde{L(p, 1)\times S^1}$. Applying the same procedure to the loop $\{x\} \times \alpha^q$ results in a 4-manifold with $\pi_1 = \Z/p\Z \oplus \Z/q\Z$. Such manifold is denoted by $\widehat{L(p, 1) \times S^1}$.

\begin{proof} These manifolds are constructed by changing the surgery coefficients on one or two of the torus surgeries used in the proof of Theorem \ref{Theorem 12} following Lemma \ref{Lemma BK}. Generalized complex manifolds with infinite cyclic fundamental group are built by \emph{not} performing one of the torus surgeries. Twisted generalized complex manifolds with finite cyclic fundamental group are built by changing one of the surgery coefficients $(1, 0, 0)$ for a $(p, 0, 0)$ with $p \neq 0$. 
The homeomorphism criteria in the infinite cyclic fundamental group case is due to Hambleton and Teichner \cite[Corollary 3]{[HT]}. If $p > 1$, then one obtains a twisted generalized complex manifold with finite cyclic fundamental groups, and its universal cover is nonspin. The 3-form $H$ is given by the generator of $H^3$. The corresponding homeomorphism criteria is due to Hambleton and Kreck, and it is given in \cite[Theorem C]{[HK]}. The case $\pi_1 = \Z/p\Z \oplus \Z/q\Z$ is left as an exercise; the homeomorphism criteria, also due to Hambleton and Kreck, is given in \cite[Theorem B]{[HK]}.\\

The argument minding the lack of a symplectic or a complex structure is verbatim to the one given for Theorem \ref{Theorem 12}.
\end{proof}

\section{Constructions of non-symplectic and non-complex generalized complex manifolds through the assemblage of symplectic ones}\label{Section 5}

As an immediate corollary to the main result of the previous section, we construct new generalized complex structures on non-spin simply connected almost-complex 4-manifolds. We also recover an existence result of Cavalcanti and Gualtieri.

\begin{proposition} {\label{Proposition 14}} \emph{(cf. \cite[Section 5]{[CG]})}. Let $n\geq m$. The manifolds
\begin{center}
$m\mathbb{CP}^2 \# n \overline{\mathbb{CP}^2}$
\end{center}
admit a generalized complex structure if and only if they admit an almost-complex one. Moreover, there exist generalized complex structures on $m\mathbb{CP}^2 \# n \overline{\mathbb{CP}^2}$ with more than one type change loci.
\end{proposition}

In particular, $m$ is an odd number. We set $m = 2g - 3$. Blowing up (Theorem \ref{Theorem 7}) and/or blowing down (Theorem \ref{Theorem 8})  $(2g - 3)(S^2\times S^2)$ produces generalized complex structures on $(2g - 3)\mathbb{CP}^2 \# n \overline{\mathbb{CP}}^2$ that can be chosen to have as many as $4 + 2g$ type change loci. As it was mentioned in Remark \ref{Remark M}, the construction can be modified to obtain a different number of loci (cf. Section \ref{Section 4.1}). For the sake of simplicity, we will give the argument for the case $g = 3$; the other cases follow verbatim. The symbol $X = Y$ indicates a diffeomorphism between manifolds $X$ and $Y$. 

\begin{proof} The necessity of the existence of an almost complex structure is clear \cite{[G]}. Let $X_3$ be the blow up of the generalized complex manifold $3(S^2\times S^2)$. Theorem \ref{Theorem 7} implies $X_3$ is generalized complex. We have
\begin{center}
$X_3 = 3(S^2\times S^2) \# \overline{\mathbb{CP}^2} = 2(S^2\times S^2) \# (S^2\times S^2\# \overline{\mathbb{CP}^2}) =$\\ $= 2(S^2\times S^2) \# \mathbb{CP}^2 \# 2 \overline{\mathbb{CP}^2} = (S^2\times S^2 \# \overline{\mathbb{CP}^2}) \#  (S^2\times S^2 \# \overline{\mathbb{CP}^2}) \# \mathbb{CP}^2 =$\\ $= (\mathbb{CP}^2 \# 2 \overline{\mathbb{CP}^2}) \# (\mathbb{CP}^2 \# 2 \overline{\mathbb{CP}^2}) \# \mathbb{CP}^2 = 3\mathbb{CP}^2 \# 4 \overline{\mathbb{CP}^2}$.
\end{center}
An iteration of the usage of blow ups and blow downs (Theorems \ref{Theorem 7} and \ref{Theorem 8} ) allows us to conclude that the manifold  $3\mathbb{CP}^2 \# n \overline{\mathbb{CP}^2}$ admits a generalized complex structure.
\end{proof}

For the remaining part of the section we proceed to consider more general production schemes of generalized complex structures on manifolds that are neither symplectic nor complex, by taking a symplectic manifold as a starting point. The following constructions are motivated by \cite{[SZs]}.

\begin{proposition} {\label{Proposition 15}} Assume $h\geq 1$. Let $X$ be a simply connected symplectic \\4-manifold that contains a symplectic surface of genus two $\Sigma \subset X$ that satisfies $[\Sigma]^2 = 0$, and $\pi_1(X - \Sigma) = 1$. The manifolds
\begin{center}
$Z_{X, h}:= X \# 2h (S^2\times S^2)$ and\\
$Z'_{X, h}:= X \# 2h (\mathbb{CP}^2 \# \overline{\mathbb{CP}^2})$
\end{center}
admit a generalized complex structure.
\end{proposition}

\begin{proof} We work out the case $h = 1$; the other cases follow verbatim from the argument. Consider the manifold $\Sigma_2 \times T^2$ endowed with the product symplectic form. Let $a_1, b_1, a_2, b_2$ be the generators of $\pi_1(\Sigma_2)$, and $x, y$ the generators of $\pi_1(T^2)$. It is easy to see that this manifold contains four pairs of homologically essential Lagrangian tori, and a symplectic surface of genus two. The tori are displayed below; the genus two surface is a parallel copy of the surface $\Sigma_2 \times \{pt\}$, and we will continue to call it $\Sigma_2$ during the proof.  According to Baldridge and Kirk \cite[Proposition 7]{[BK3]}, the fundamental group
\begin{center}$\pi_1(\Sigma_2\times T^2 - (\Sigma_2 \cup T_1 \cup \cdots \cup T_4))$\end{center}
is generated by the loops  $x, y, a_1, b_1, a_2, b_2$. Moreover,  with respect to certain paths to the boundary of the tubular neighborhoods of the $T_i$ and $\Sigma_2$, the meridians and two Lagrangian push offs are given by
\begin{itemize}
\item $T_1: m_1 = x, l_1 = a_1, \mu_1= [b^{-1}, y^{-1}]$,
\item $T_2:  m_2 = y, l_2 = b_1a_1b^{-1}, \mu_2 = [x^{-1}, b_1]$,
\item $T_3:  m_3 = x, l_3 = a_2, \mu_3 = [b_2^{-1}, y^{-1}]$,
\item $T_4: m_4 = y, l_4 = b_2a_2b_2^{-1}, \mu_4 = [x^{-1}, b_2]$,
\item $\mu_{\Sigma_2} = [x, y]$.
\end{itemize}

The loops $a_1, b_1, a_2, b_2$ lie on the genus 2 surface and form a standard set of generators; $[a_1, b_1][a_2, b_2] = 1$ holds.\\

Build the symplectic sum \cite[Theorem 1.3]{[Go]} \begin{center} $Z:= X\#_{\Sigma = \Sigma_2} \Sigma_2\times T^2$.\end{center}
Given that the loops $a_1, b_1, a_2. b_2$ lie on $\Sigma_2 \times \{x\} \subset \Sigma_2 \times T^2$ for $x\in T^2$, and the meridian is given by $\mu_{\Sigma_2} = [x, y]$, using van-Kampen's Theorem we see that our hypothesis $\pi_1(X - \Sigma) = 1$ implies $\pi_1(Z) = \Z x \oplus \Z y$.
Perturb the symplectic form so that the tori $\{T_1, T_2, T_3, T_4\}$ become symplectic \cite[Lemma 1.6]{[Go]}. Perform simultaneously a $(1, 0, 0)$-surgery on each of the tori to obtain a manifold $Z_{X, 2}$. By Theorem \ref{Theorem 6}, the manifold $Z_{X, 2}$ admits a generalized complex structure, which has four type change loci.
The surgeries set $x  = 1 = y$. Thus, $\pi_1(Z_{X, 2}) = 1$.
To conclude on the diffeomorphism type of $Z_{X, 2}$, we observe the effect that the torus surgeries has on the embedded submanifolds. The tori $y\times b_2$ and $y\times a_2$ become embedded 2-spheres, each of which is a factor of an $S^2\times S^2$ summands. Thus, $Z_{X, 2}$ is diffeomorphic to $X \# 2(S^2\times S^2)$.
In order to construct $Z_{X, h}$ for $h\geq 2$, one builds the symplectic sum \begin{center} $Z:= X \#_{\Sigma = \Sigma_2} \Sigma_2 \times \Sigma_h$\end{center}
and submits it to surgical procedure similar to the one described above, making use of Proposition \ref{Proposition 11}. The generalized complex manifold $Z'_{X, h}$ is obtained by blowing up at a nondegenerate point on $Z_{X, h}$ (Theorem \ref{Theorem 7}), and then blowing it down (Theorem \ref{Theorem 8}).

\end{proof}

To finalize this section,  we produce generalized complex structures on more general connected sums.

\begin{proposition} {\label{Proposition 16}} Suppose $h\geq 2$. Let $X$ and $Y$ be symplectic 4-manifolds that contain symplectic tori $T_X \subset X,  T_Y \subset Y$ such that $[T_X]^2 = 0 = [T_Y]^2$, and\\ $\pi_1(X) = \pi_1(X - T_X) = \pi_1(Y) = \pi_1(Y - T_Y) = 1$. The manifolds
\begin{center}
$Z_{X,Y, 2h - 1}:= X \# (2h - 1) (S^2\times S^2) \# Y$ and
$Z'_{X,Y, 2h - 1}:= X \# (2h - 1) (\mathbb{CP}^2 \# \overline{\mathbb{CP}^2}) \# Y$
\end{center}
admit a generalized complex structure.
\end{proposition}

The following proof was inspired by an argument due to Z. Szab\'o given in \cite{[SZs]}. 

\begin{proof} Following Proposition \ref{Proposition 10}, perturb the symplectic sum \cite[Lemma 1.6]{[Go]} so that the tori $\{T_1, T_4, T_5, T_6, T_7, T_8\}$ become symplectic. Build the symplectic sum \cite[Theorem 1.3]{[Go]}
\begin{center}
$Z:= X \#_{T_X = T_1} \Sigma_2\times \Sigma_2 \#_{T_4 = T_Y} Y$.
\end{center}

Notice that in $\pi_1(Z)$, we have $a_1 = c_1 = a_2 = c_2 = 1$. Simultaneously perform $(0, 1, 0)$-surgeries on the tori $T_5, T_6, T_7$, and $T_8$. Denote the resulting generalized complex 4-manifold by $Z_{X, Y, 3}$ (Theorem \ref{Theorem 6}). This manifold contains four type change loci. 

Let us take  a look at the surviving submanifolds of $Z_{X, Y, 3}$. The torus surgeries performed along $b_1$ and $b_2$ turn the tori $b_1\times d_1, b_2\times d_1$, and the genus 2-surface $\{x\}\times \Sigma_2$ into imbedded 2-spheres of self-intersection zero. Each of these spheres is a factor of an $S^2\times S^2$-summand. Therefore,
\begin{center}
$Z_{X, Y, 3} = X \# 3(S^2\times S^2) \# Y$.
\end{center}
The generalized complex manifold constructed is neither symplectic \cite{[Ta1]} nor complex \cite{[Ko], [BHPV]}. The generalized complex manifold $Z'_{X, Y, 3} = X \#3(\mathbb{CP}^2 \# \overline{\mathbb{CP}^2}) \# Y$ is obtained by blowing up a nondegenerate point  on  $Z_{X, h}$ (Theorem \ref{Theorem 7}), and then blowing it down (Theorem \ref{Theorem 8}).
\end{proof}

\begin{remark} {\label{Remark 3}} Submanifolds of higher genus. \emph{The reader will notice that the symplectic sums involved in the previous arguments can be modified in order to yield similar results in the case the starting manifold $X$ contains symplectic submanifolds of higher genus}.
\end{remark}

\section{Non-simply connected twisted generalized complex 4-manifolds}{\label{Section 6}}

Given that our principal mechanism of construction is to apply torus surgeries to non-simply connected building blocks, an organic next step is the study of existence of twisted generalized complex structures in the non-simply connected realm.

\subsection{Finitely presented fundamental groups} \label{Section 6.1} In his lovely paper \cite{[Go]}, Gompf proved that every finitely presented group is the fundamental group of a symplectic 4-manifold. Given that the existence of a symplectic structure naturally induces the existence of a generalized complex structure (Example \ref{Example A}), \emph{a posteriori} the fundamental group imposes no restriction on the existence of such a structure.\\

Moreover, Kotschick \cite{[DK]} proved that any finitely presented group occurs as the fundamental group of an almost-complex 4-manifold. It is an interesting question, how much the size of a generalized complex 4-manifold depends on its fundamental group. In this direction we obtain the following result.

\begin{theorem} {\label{Theorem 17}} Let $G$ be a finitely presented group with a presentation consisting of $g$ generators $x_1, \ldots, x_g$, and $r$ relations $w_1, \ldots, w_r$, and let $k$ be a nonnegative integer. There exists a non-symplectic twisted generalized complex 4-manifold $X(G, k)$ with \begin{center}$\pi_1(X(G, k)) \cong G$, $e(X(G, k)) = 4(g + r + 2) + k$, and $\sigma(X(G, k)) = - k$.\end{center}
\end{theorem}

We point out that the manifolds of Theorem \ref{Theorem 17} do not share the homotopy type of a complex surface \cite{[BHPV]}. The remaining part of the section is technical; see also Remark \ref{Remark Y}.\\

In order to prove the theorem, we follow an argument due to Baldridge and Kirk. In \cite{[BK]}, these authors build a symplectic manifold $N$, which serves as a fundamental building block for these constructions since it allows the manipulation on the number of generators and relations on fundamental groups. We proceed to describe it.\\

Begin with a 3-manifold $Y$ that fibers over the circle, and build the 4-manifold\begin{center} $N: = Y\times S^1$.\end{center} Its Euler characteristic and its signature are both zero, and it admits a symplectic structure \cite[p. 856]{[BK]}.\\

The fundamental group of $N$ has the following presentation.
\begin{center}
$\pi_1(N) = \big<H, t \big| R^g_{\ast}(x) = t x t^{-1}, x\in H\big> \times \big<s\big>$,
\end{center}
where the group $H$ has a presentation given by
\begin{center}
$H = \big<x_{1,1}, y_{1,1}, \ldots, x_{g, 1}, y_{g, 1}, x_{1, 2},$ $y_{1, 2}, \ldots, x_{g, 2}, y_{g, 2}, \ldots, x_{1, n}, y_{1, n}, \ldots , x_{g, n}, y_{g, n} \big| \Pi_{i = 1}^n \Pi_{j = 1}^g [x_{j, i}, y_{j, i}] \big>$.\\
\end{center}

Let $G$ be a finitely presented group with $g$ generators and $r$ relations. The fundamental group $\pi_1(Y\times S^1)$ has classes $s, t, \gamma_1, \ldots, \gamma_{g + r}$ so that
\begin{center}
$G \cong \pi_1(Y\times S^1) / N(s, t, \gamma_1, \ldots, \gamma_{g + r})$,
\end{center}

where $N(s, t, \gamma_1, \ldots, \gamma_{g + r})$ is the normal subgroup generated by the aforementioned classes.\\

The symplectic manifold $N$ contains $g + r + 1$ symplectic tori \begin{center}$T_0^Y, T_1^Y, \ldots, T_{g + r}^Y\subset N$\end{center} of self-intersection zero that have two practical features regarding our  fundamental group computations.
\begin{itemize}

\item The generators of $\pi_1(T_0^Y)$ represent $s$ and $t$, and
\item the generators of $\pi_1(T_i^Y)$ represent $s$ and $\gamma_i$.

\end{itemize}

The curve $s$ has the form $\{y\} \times S^1 \subset Y\times S^1$, with $y\in Y$, and $\gamma_i$ has the form $\gamma_i \times \{x\}\subset Y\times \{x\}, x\in S^1$.  The role of the curves $\gamma_i$ is to provide the $r$ relations in a presentation of $G$ \cite[Section 4]{[BK]}; we point out that the number of curves $\gamma_i$ is greater than $r$.\\

Serious technical issues arise during cut-and-paste constructions in the computation of fundamental groups, the choice of basepoint being the culprit. The argument used here to prove Theorem \ref{Theorem 17} builds upon the careful analysis done by J. Yazinski in \cite[Section 4]{[Y]}.\\

Let $R$ be the symplectic manifold obtained from $\Sigma_2\times \Sigma_2$ by applying the following five Luttinger surgeries. Perform two $(1, 0, -1)$-surgeries on the tori $T_5, T_6$, and three $(0, 1, -1)$ -surgeries on the tori $T_3, T_7, T_8$. The manifold $R$ contains the homologically essential Lagrangian tori  $T_1, T_2, T_4 \subset R$. Perturb its symplectic form so that all these tori become symplectic \cite[Lemma 1.6]{[Go]}.\\

Let us pin down a basepoint for the fundamental group computations that are to come. Set the basepoint to be $v\in \partial(\nu(T_1))$, where $\nu(T_1)$ is a tubular neighborhood of $T_1$. The fundamental group $\pi_1(R - T_1)$ is generated by the elements $\{a_1, b_1, \ldots, c_2, d_2, \epsilon_1, \epsilon_2, \ldots, \epsilon_m\}$, where the elements $\{\epsilon_j\}$ are conjugates to the based meridian $\mu_1 = [b_1^{-1}, d^{-1}]$. Let $\beta$ be a path inside $R - T_1$ from the basepoint $(x, y)$ that was chosen to calculate $\pi_1$ of the complement of the tori inside $\Sigma_2\times \Sigma_2$ (cf. \cite{[FPS]} and Proposition \ref{Proposition 10}) to $v$, which is contained in the torus $b_1\times d_1$ and such that it intersects $T_1$ transversally at one point. The paths representing generators of $\pi_1(\Sigma_2\times \Sigma_2, (x, y))$ can be conjugated by $\beta$ in order to obtain paths that are based at $v$. The relations stated in Proposition \ref{Proposition 10} continue to hold in $\pi_1(R - T_1, v)$ under the new choice of basepoint.

\begin{proof} The task at hand is to construct a twisted generalized complex manifold $X(G)$ that is not symplectic, such that $\pi_1(X(G)) = G, e(X(G)) = 4(g + r + 2)$, and $\sigma(X(G)) = 0$. This is done by applying torus surgeries to a symplectic sum along tori composed of  $Y$, a symplectic manifold built from $\Sigma_2\times \Sigma_3$, and $g + r + 1$ copies of $R$, which we will denote by $R_i, i = 1, \ldots, g + r$. The symplectic tori inside each copy are denoted by $T_{1, i}, T_{2, i}, T_{4, i} \subset R_i$. Once $X(G)$ is obtained, applying blow ups to it results in the manifold $X(G, k)$. The details are as follows.

Consider the manifold $\Sigma_2\times \Sigma_3$ equipped with the product symplectic form. A small modification on the fundamental group computations in \cite{[MCH]} and Proposition \ref{Proposition 11}  yields that 
the Lagrangian tori inside $\Sigma_2\times \Sigma_3$, the meridians and the Lagrangian pushoffs of twelve tori are given by\\
\begin{center}
$T_1 := a_1\times c_1, m_1 = a_1, l_1 = c_1, \mu_1 = [b_1^{-1}, d_1^{-1}]$,\\
$T_2 : = a_1\times c_2,  m_2 = a_1, l_2  = c_2, \mu_2 = [b_1^{-1}, d_2^{-1}]$,\\
$T_3 := a_1\times c_3, m_3 = a_1, l_3 = c_3, \mu_3 = [b_1^{-1}, d_3^{-1}]$,
$T_4 := a_2\times c_1, m_4 = a_2, l_4 = c_1, \mu_4 = [b_2^{-1}, d_1^{-1}]$,\\
$T_5 := a_2\times c_2, m_5 = a_2, l_5 = c_2, \mu_5 = [b_2^{-1}, d_2^{-1}]$,\\
$T_6 := a_2\times c_3, m_6 = a_2, l_6 = c_3, \mu_6 = [b_2^{-1}, d_3^{-1}]$,\\
$T_7 := b_1\times c_1, m_7 = b_1, l_7 = d_1 c_1 d_1^{-1}, \mu_7 = [a_1^{-1}, d_1]$,\\
$T_8 := b_2\times c_2, m_8 = b_2, l_8 = d_2 c_2 d_2^{-1}, \mu_8 = [a_2^{-1}, d_2]$,\\
$T_9 := a_2\times d_1, m_9 = b_2 a_2 b_2^{-1}, l_9 = d_1, \mu_9 = [b_2, c_1^{-1}]$,\\
$T_{10} := a_1\times d_2,  m_{10} = b_1 a_1 b_1^{-1}, l_{10} = d_2, \mu_{10} = [b_1, c_2^{-1}]$,\\
$T_{11} := a_1\times d_3, m_{11} = b_1 a_1 b_1^{-1}, l_{11} = d_3, \mu_{11} = [b_1, c_3^{-1}]$, and\\
$T_{12} := a_2\times d_3, m_{12} = b_2 a_2 b_2^{-1}, l_{12} = d_3, \mu_{12} = [b_2, c_3^{-1}]$.
\end{center}

Construct a symplectic manifold $S$ by applying Luttinger surgeries and perturbing the symplectic form as follows. Perform $(1, 0, -1)$-surgery on the tori $T_1, T_2, T_4, T_5, T_7, T_8$, and on the tori $T_9, T_{10}, T_{11}$ perform $(0, 1, -1)$-surgeries.  By \cite{[ADK]}, the resulting manifold is symplectic. Now, perturb the symplectic form so that the tori $T_3, T_{12} \subset S$ become symplectic \cite[Lemma 1.6]{[Go]}. Notice that $H_1(S; \Z) \cong \Z$ is generated by $c_3$.

Regarding $\pi_1(S - T_{12})$, set the basepoint to be $v_{12} \in \partial (\nu(T_{12}))$. Let $\beta_{12}$ be a path in $S - \nu(T_{12})$ from the chosen basepoint $(x_2, y_3)$ of $\pi_1(\Sigma_2\times \Sigma_3, (x_2, y_3))$ used to compute the fundamental group of the complement of the tori to $v_{12}$, which is contained in the torus $b_2 \times c_3$ and intersects $T_{12}$ transversally in one point. Conjugating the paths that represent generators of $\pi_1(\Sigma_2\times \Sigma_3, (x_2, y_3))$ by $\beta_{12}$ yields paths based at $v_{12}$. The fundamental group $\pi_1(S - T_{12})$ is generated by the elements $\{a_1, b_1, \ldots, c_3, d_3, \epsilon_1^S, \ldots, \epsilon_m^S\}$, where the elements $\{\epsilon_j^S\}$ are conjugates to the based meridian $\mu_{12}$.

Following \cite{[Go]} and \cite{[BK]}, we start by building a symplectic manifold that provides us with the generators of $G$ (plus some generators that we will be getting rid of at a later stage with the usage of torus surgeries), and that contains enough submanifolds to introduce the relations of $G$ in a later step. For these purposes, build the symplectic sum \cite[Theorem 1.3]{[Go]}

\begin{center}
$Z:= S \#_{T_{12} = T_0^Y} Y$
\end{center}

using the diffeomorphism $\phi: T_{12} \rightarrow T_0^Y$ that induces the identifications $a_2\mapsto s$, $d_3\mapsto t$ on fundamental groups. The meridian $\mu_0^Y$ is identified with the based meridian $\mu_{12} = [b_2, c_3^{-1}]$. Let $z$ be the basepoint for the block $Z$, and let $\eta$ be a path that takes $z$ to $\phi(v_{12})$ with $v_{12} \in \partial (\nu(T_{12}))$. Conjugation by $\eta$ provides us with paths that represent elements in $\pi_1$. We abuse notation, and we keep calling the conjugated elements with the same symbols. 

Using Seifert-van Kampen's Theorem, the fundamental group is given by
\begin{center}
$\pi_1(Z) =  \frac{\pi_1(Y - \nu(T_0^Y)) \ast \pi_1(S - \nu(T_{12}))}{<a_2 s^{-1},  d_3 t^{-1},  \mu_0^Y (\mu_{12})^{-1}>}$.
\end{center}

Here, $\nu(T_0^Y), \nu(T_{12})$ denote tubular neighborhoods of the corresponding tori; in particular, these manifolds intersect in an open neighborhood of $T^3$.

Moreover, the tori $T_1^Y, \ldots, T^Y_{r + g}$ are contained in $Z$. The next step is to use these tori to perform $r + g$ symplectic sums with the purpose of  introducing the $r$ relations in the presentation of $G$. In the process, more relations are introduced into the group presentation. We will get rid of any extra relation by using torus surgeries at a latter step. Take $i = 1, \cdots, r + g$ copies of $R$, and call each copy $R_i$. We start by choosing based curves in $R_i$ that are homotopic to the generators $a_{1, i}$, and $c_{1, i}$ for all $i$. By Section \ref{Section 4.1}, the based meridian $\mu_{1,i}$ is taken to be the commutator $[b_{1, i}^{-1}, d_{1, i}^{-1}]$. These push offs are contained in the boundary of the tubular neighborhood of the respective torus inside $R_i$.

Minding the choice of base point for the application of Seifert-van Kampen's Theorem, we will follow the discussion that precedes the proof for each copy of $R$. So, we have basepoints $v_i \in \partial(\nu(T_{1, i}))$ and conjugating paths $\beta_i$ .

Let $Z':= Z - (\nu(T_1^Y)\cup \cdots \cup \nu(T_{r + g}^Y))$, where $\nu(T_i^Y)$ denotes a tubular neighborhood of the symplectic torus $T_i^Y\subset Z$. Build the symplectic sums of $Z$ and $R_i$ along $T_{1, i}$ using the diffeomorphism $\phi_i: T_{1, i} \rightarrow T_i^Y$ that identify the generators
\begin{center}
$a_{1, i}\mapsto s_i$\\
$c_{1, i}\mapsto \gamma_i$
\end{center}

to build a symplectic manifold $\hat{X}$. The tori $T_{2, i}, T_{4, i}, T_3'$ are contained in $\hat{X}$. Simultaneously perform $(0, 1, 0)$-, $(1, 0, 0)$-, and $(0, 1, 0)$-surgery respectively on these tori to set $ c_{2, i} = a_{2, i} = c_3 = 1$. Denote the resulting twisted generalized complex manifold by $X(G)$. We claim that the fundamental group of $X(G)$ is given by the finitely presented group $G$ of our hypothesis.

We need to make sure that the usage of Seifert-van Kampen's Theorem is done carefully. There must be a common basepoint that is being shared by the open neighborhoods of the pieces being glued together. Set $z_0$ to be the basepoint in $Z'$. Regarding the basepoints $v_i \in \partial (\nu(T_{1, i}))$ for the pieces $R_i - T_{1, i}$, fix paths $\eta_i$ that take $z_0$ to $\phi(v_{12})$. Take $\overline{R_i - \nu(T_{1, i})}$ to be the union of $R_i - \nu(T_{1, i})$ and a neighborhood of the path $\eta_i$. We continue to abuse notation, and denote an element of $\pi_1(\overline{R_i - \nu(T_{1, i}})$ with the same symbol for the related element in $\pi_1(R_i - \nu(T_{1, i}))$, since conjugation with the path $\eta_i$ correlates the elements and the relations continue to hold.

By Seifert-van Kampen's Theorem we have

\begin{center}
$\pi_1(X(G)) \cong \frac{\pi_1(Z') \ast \pi_1(\overline{R_1 - \nu(T_{1, 1})}) \ast \cdots \ast \pi_1(\overline{R_{r + g} - \nu(T_{1, r + g})})}{N}$,\\
\end{center}

where $N$ is the normal subgroup of
 $\pi_1(Y) \ast \pi_1(S) \ast \pi_1(R_1 - T_{1, 1}) \ast \cdots \ast \pi_1(R_{r + g} - T_{1, r + g})$ 
generated by $a_{1, i}\overline{s}_i^{-1}, c_{1, i} \overline{\gamma}_i^{-1}$,the meridians $\mu'_{12} {\mu_{0}^Y}^{-1}, \mu^Y_i {\mu_{1, i}}^{-1}$ and their conjugates, and $c_{2, i} = a_{2, i} = c_3 = 1$.  The element $\bar{\gamma}_i$ is a based loop in $Z'$ that is freely homotopic to a push off of the loop $\gamma_i \times \{x\}\subset \gamma_i \times S^1$ to the boundary of a tubular neighborhood of the torus $T_i = \gamma_i \times S^1$. The element $\mu_i^Y$ is a based meridian corresponding o the torus $T_i \subset Z'$. Given that the curves $\bar{s}_i$ and $s$ are of the form $\{x\}\times S^1$, the elements $\bar{s}_i$ are based curves in $Z'$ that are homotopic to $s$ on the boundary of a tubular neighborhood of $T_i = \gamma_i \times S^1$ (cf. \cite[Proof of Lemma 4.3]{[Y]}).




We need to see that the elements $a_1, b_1, a_2, b_2, c_1$, $d_1, c_2, d_2, d_3^{-1}t, b_{1, i}, b_{2, i}, d_{1, i}, d_{2, i}$ (and the conjugates of the meridians $\epsilon_{1, i}, \ldots, \epsilon_{n, i}, \epsilon_1^S, \ldots, \epsilon_m^S$) for $i = 1, \dots, g + r$ are trivial in $\pi_1(X(G))$. The $(1, 0, 0)$- and $(0, 1, 0)$-torus surgeries kill the generators $b_{1, i}, b_{2, i}, d_{1, i}, d_{2, i}$. This can be seen by plugging in any of the identities $c_{2, i} = 1 = a_{2, i}$ in the presentation of $\pi_1(R_i)$ (see Proposition \ref{Proposition 10}). Moreover, the triviality of the three generators $c_{2, i} = a_{2, i} = c_3 = 1$ implies that the meridians $\mu_{12}, \mu_0^Y, \mu^Y_i, \mu_{1, i}$ and their conjugates are trivial in $\pi_1(X(G))$. Using the identity $c_3 = 1$ in the relations in $\pi_1(S)$, we see that $a_1 = b_1 = a_2 = b_2 = c_1  = d_1 = c_2 = d_2 = d_3 = 1$, and $t = 1$ due to the identifications of the generators of the tori during the construction of the symplectic sum $Z$. Moreover, since $\mu_i^Y = 1$ for $i = 1, \ldots, r + g$, their conjugates are trivial. This implies that $\bar{\gamma}_i$ is conjugate to $\gamma_i$ in $\pi_1(X(G))$. By the choice of $\gamma_i$ \cite{[BK]}, we have $\pi_1(X(G)) \cong G$.\\

To argue that the manifold $X(G)$ does not admit a symplectic structure, we proceed as follows. The torus $T_6'$ is contained in $X(G)$, since it was disjoint from all the tori involved in the surgeries. The $(0, 1, 0)$-surgery that was applied on $T_3'$ turns $T_6'$ into an embedded 2-sphere of self-intersection zero. By the Adjunction inequality (Theorem \ref{Theorem 9}), the manifold $X(G)$ has trivial Seiberg-Witten invariants. Taubes' work \cite{[Ta1]} implies that $X(G)$ does not admit a symplectic structure, as claimed.\\

The computations of the characteristic numbers of $X(G)$ are straight-forward.
\end{proof}

\begin{remark}{\label{Remark Y}} Another construction of a manifold $X(G)$. \emph{A twisted generalized complex 4-manifold that is neither symplectic, nor complex with arbitrary finitely presented fundamental group, Euler characteristic $10 + 4(g + r + 1)$ and signature $-1$ can be produced as follows. The symplectic manifold in the main result of \cite{[Y]} contains a symplectic 2-torus whose meridian in the complement is trivial. Use this symplectic 2-torus to form a symplectic sum \cite{[Go]} of the manifold $X$ in \cite[Theorem 1.1]{[Y]} with $\Sigma_2\times \Sigma_2$, and then apply accordingly six surgeries of Theorem \ref{Theorem 6} cf. Section \ref{Section 4.1}. Such a twisted generalized complex structure contains six type change locus cf. Remark \ref{Remark M}.}.
\end{remark}

\subsection{Abelian groups} \label{Section 6.2} It is of our interest to provide explicit constructions, which we hope will be useful in improving our understanding of twisted generalized complex structures.  We begin this section with two examples that expose different twisted generalized complex structures on the same 4-manifold.

\begin{example} {\label{Example 18}} \emph{Twisted generalized complex structure on $S^3\times S^1$ with a single type change locus \emph{(cf. \cite[Example 6.38]{[G1]}, \cite[Example 4.1]{[CG]})}}. Consider $T^2\times S^2$ endowed with the product symplectic form. A $(1, 0, 0)$-torus surgery on the symplectic submanifold $T^2 \times \{s\}$ for $s\in S^2$ along one of the loops carrying a generator of $\pi_1(T^2\times S^2)$ results in $S^3\times S^1$. The diffeomorphism is established as follows. The handlebody of $T^2\times S^2$ is given at the top of Figure 1. We are using the 1-handle notation introduced by Akbulut in \cite{[A0]}. The effect that a $(1, 0, 0)$-torus surgery has on the handlebody decomposition of $T^2 \times D^2$ is to interchange a 1-handle (dotted circle) and a 0-framed 2-handle \cite[Figure 8.25, p.316]{[GS]}, \cite[Figure 17]{[AY]}. Thus, performing a $(0, 0, 1)$-torus surgery on $T^2\times \{pt\}$ results in the second diagram of Figure 1. The second and third rows show the isotopies and handle cancellations that establish a diffeomorphism with $S^3\times S^1$. By Theorem \ref{Theorem 6}, this primary Hopf surface has a twisted generalized complex structure with one type change locus.

\begin{figure}{\label{Figure 1}}
\begin{center}
\includegraphics[scale=0.5]{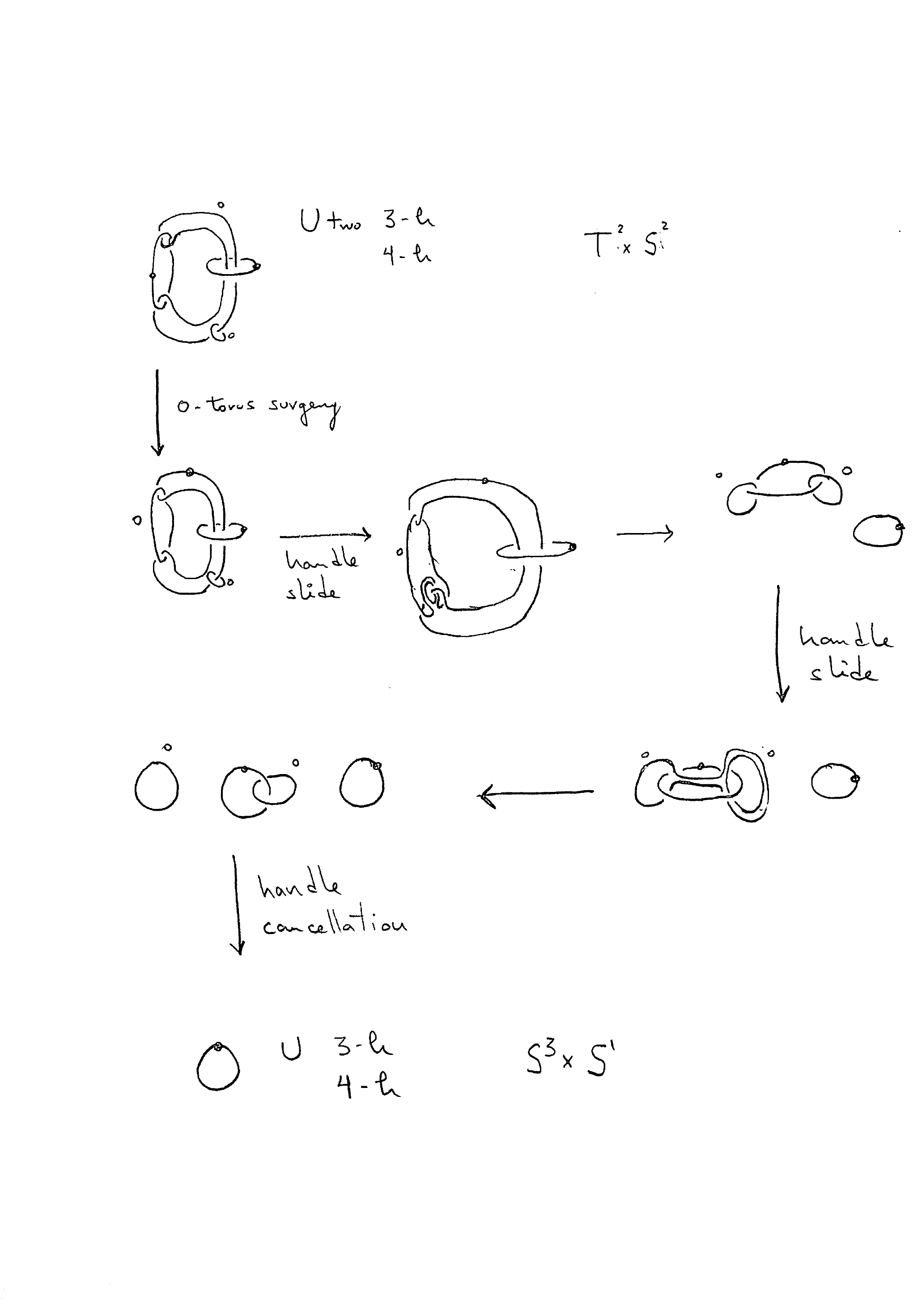}
\caption{The result of performing a $(1, 0, 0)$-surgery to $T^2\times S^2$ on the 2-torus $T^2\times \{pt\}$ is $S^3\times S^1$}
\end{center}
\end{figure}

\end{example}

\begin{example} {\label{Example 19}}\emph{Twisted generalized complex structure on blow-ups of $L(p, 1)\times S^1$ with one, two or three type change loci}. To build a twisted generalized complex structure on this secondary Hopf surface, start with the  4-torus $T^4 = T^2\times T^2$ endowed with the product symplectic form. Denote by $x, y, a, b$  both the generators of $\pi_1(T^2\times T^2) = \Z x \oplus \Z y \oplus \Z a \oplus \Z b$ and the corresponding loops as well. Under this notation, the tori \begin{center} $T_1:= x \times a, T_2:= y\times a$ \end{center}and their respective dual tori $y\times b, x\times b$ are Lagrangian, and the torus $T_3:= a\times b$ and its dual $x\times y$ are symplectic.\\

According to Baldridge and Kirk \cite[Section 2]{[BKS]} and \cite[Theorem 1]{[BK3]}, the fundamental group of $T^4 - (T_1 \cup T_2)$ is generated by the loops $x, y, a, b$ and the relations $[x, a] = [y, a] = 1$ hold. The meridians of the tori and the two Lagrangian push offs of their generators are given by the following formulae:\\

\begin{center}
$m_1 = x, l_1 = a, \mu_1 = [b^{-1}, y^{-1}]$ and\end{center}
\begin{center}
$m_2 = y, l_2 = bab^{-1}, \mu_2 = [x^{-1}, b]$.
\end{center}

Apply a $(p, 0, -1)$-torus surgery on $T_1$. This Luttinger surgery introduces the relation $x^p = [b^{-1}, y^{-1}]$. Now, do a  $(1, 0, 1)$-surgery on $T_2$ to introduce the relation $y = [x^{-1}, b]$. Last, perform a $(0, 1, 0)$-torus surgery on $T_3$ along $b$, and call the resulting manifold $X_p(1)$. The last surgery sets $b = 1$, which implies $y = [x^{-1}, 1] = 1$, and $x^p = 1$. Since the elements $x$ and $a$ commute, we have\\

\begin{center} $\pi_1(X_p(1)) = < x, a: x^p = 1, [x, a] = 1 > \cong  \Z/p\Z \oplus \Z$.\end{center}

Theorem \ref{Theorem 6} implies that $X_p(1)$ admits a twisted generalized complex structure with one single type change locus. If we simultaneously perform $(1, 0, 0)$-surgery on a symplectic $T_2$, and $(0, 1, 0)$-torus surgery on $T_3$, a twisted generalized complex structure $(X_p(2), H, \mathcal{J})$ with two type change loci. Performing simultaneously the surgery of Theorem \ref{Theorem 6} on three symplectic $T_1, T_2, T_3$ yields a third twisted generalized complex structure $(X_p(3), H, \mathcal{J})$ with three type change loci.\\

Set $X_p:= X_p(i)$ for $i\in \{1, 2, 3\}$. Let us check that there exists a diffeomorphism $X_p \rightarrow L(p, 1)\times S^1$, by using the fact that a diffeomorphism between two 3-manifolds extends to a diffeomorphism between the 4-manifolds obtained by taking the product of $S^1$ with the diffeomorphic 3-manifolds. We are grateful to Ronald J. Stern, having learned the following argument from him. \\

Think of the 4-torus as $T^4 = T^3 \times S^1 = (x \times y \times b) \times a$. The torus surgeries performed to obtained $X_p$ can be seen as Dehn surgeries surgeries on $D^2\times T^2 = (D^2\times S^1) \times S^1$ along the curves $x, y, b$ in $T^3$ (see Section \ref{Section 3.1}). In particular, we have that the twisted generalized complex manifold $X_p = \widetilde{T^3}\times S^1$, where $\widetilde{T^3}$ is the 3-manifold obtained from applying the three Dehn surgeries on the 3-torus.\\

The diffeomorphism between $\widetilde{T^3}$ and $L(p, 1)$ can be seen through the analysis of the effect that the Dehn surgeries have on the handlebody of $T^3$. The 3-torus is obtained by 0-surgery on the Borromean rings \cite[Figure 5.25, p. 159]{[GS]}, and the Lens spaces are obtained by performing a $-p$-surgery on the unknot \cite[Figure 5.24, p.158]{[GS]}. Thus, $X_p$ is diffeomorphic to $L(p, 1)\times S^1$.\\

By blowing up k non-degenerate points and iterating the usage of Theorem \ref{Theorem 7}, one concludes that \begin{center} $L(p, 1)\times S^1 \# k  \overline{\mathbb{CP}^2}$\end{center} admits three twisted generalized complex structures with one, two, and three type change loci.
\end{example}

\begin{remark} {\label{Remark 4}} New twisted generalized complex structures on blow-ups of $S^3\times S^1$ and $T^2\times S^2$. \emph{Since there are diffeomorphisms $L(1, 1) = S^3$ and $L(0, 1) = S^1\times S^2$, Example \ref{Example 19} covers the existence of twisted generalized complex structures on blow-ups of $S^3\times S^1$ and on blow ups of $T^2\times S^2$ as well. In particular, the twisted generalized complex structure on $S^3\times S^1$ has two type change loci. The twisted generalized complex structure on the Hopf surface displayed in Example \ref{Example 18} has a single type change locus. Furthermore, it was proven above that there are twisted generalized complex structures on $T^2\times S^2$ with one or two type change loci}.\\
\end{remark}

The main result of this section is the following theorem.

\begin{theorem} {\label{Theorem 20}}The manifolds
\begin{itemize}
\item $\hat{m}(S^2\times S^2) \# S^3\times S^1$,  and
\item $\hat{m} \mathbb{CP}^2 \# \hat{n}\overline{\mathbb{CP}^2} \# S^3\times S^1$
\end{itemize}
admit a twisted generalized complex structure if and only if they admit an almost complex structure.
\end{theorem}

The production of these manifolds involve torus surgeries on the product $\Sigma_2\times \Sigma_g$ of a genus two surface and a genus $g$ surface, as it was done for Theorem \ref{Theorem 12} in Section \ref{Section 4.2}. In particular $\hat{m} = m + 1 = 2g - 2$. For the number of type change loci see Remark \ref{Remark M} and Section \ref{Section 4.1}.

\begin{proof} The necessity of the existence of an almost-complex structure is clear \cite{[G]}. The construction of the manifolds claimed by the first item in the statement of the theorem was carried out in Proposition \ref{Proposition 13}. To see that these manifolds are reducible, one proceeds as in the proof of Theorem \ref{Theorem 12}.
The existence of the manifolds of the second claim in the statement  follows by applying blow ups and blow downs (Theorem \ref{Theorem 7} and Theorem \ref{Theorem 8}) to the manifolds of the first item (compare with the proof of Proposition \ref{Proposition 14}).
\end{proof}

The symmetric product of a genus $g$ surface $Sym^2(\Sigma_g)$  is a surface of general type obtained by taking the quotient of the product of two surfaces of genus $g$ $\Sigma_g\times \Sigma_g$ by the action of the involution $\Sigma_g\times \Sigma_g \rightarrow \Sigma_g\times \Sigma_g$ defined by $(x, y) \mapsto (y, x)$. 
By applying the procedure described in the previous sections to this K\"ahler manifold one obtains the following proposition.

\begin{proposition} {\label{Proposition 21}} Let $k_1, k_2, k_3, k_4$ be a nonnegative integers. There are twisted generalized complex non-symplectic, non-complex 4-manifolds with for the following homeomorphism classes

\begin{itemize}
\item $\pi_1 = 1:  (g^2 - 3g + 1) \mathbb{CP}^2 \# k_1 \overline{\mathbb{CP}^2}$,

\item $\pi_1 = \Z/p\Z:  (g^2 - 3g + 1) \mathbb{CP}^2 \# k_2 \overline{\mathbb{CP}^2}\# \widetilde{L(p, 1)\times S^1}$

\item $\pi_1 = \Z/p\Z \oplus \Z/q\Z:  (g^2 - 3g + 1) \mathbb{CP}^2 \# k_3 \overline{\mathbb{CP}^2}\# \widehat{L(p, 1)\times S^1}$

\item $\pi_1 = \Z:  (g^2 - 3g + 2) \mathbb{CP}^2 \# k_4 \overline{\mathbb{CP}^2} \# S^3\times S^1$

\end{itemize}
\end{proposition}

We finish the section by mentioning that it is straight-forward to obtain similar existence results of twisted generalized complex manifolds that do not admit a symplectic nor a complex structure, and whose fundamental group is among the following choices

\begin{center}
$\pi_1 = \Z/p_1\Z \oplus \Z/p_2\Z \oplus \cdots \oplus \Z/p_{2g - 1}\Z$.
\end{center}

\subsection{Free and surface groups} \label{Section 6.3} For what follows we start with the symplectic manifold $T^2\times \Sigma_g$ endowed with the product form. 

\begin{theorem} {\label{Theorem 22}}Let $g\geq 2$, and assume k to be a nonnegative integer. There exist twisted generalized complex non-symplectic 4-manifolds $X_{F, g, k}(i)$ and $X_{S, g, k}(j)$, which have $e(X_{F, g, k}) = e(X_{S, g, k}) = k$ and  $\sigma(X_{F, g, k}) = \sigma(X_{S, g, k}) = -k$ such that
\begin{itemize}
\item $\pi_1(X_{F, g, k}(i)) = \overbrace{\Z \ast \cdots \ast \Z}^g$
\item $\pi_1(X_{S, g, k}(j)) = \pi_1(\Sigma_g)$.
\end{itemize}
The twisted generalized complex structures on these manifolds can be chosen to have a number of $i\in \{1, \ldots, 2 +g\}$ and $j\in \{1, 2\}$ type change loci respectively.
\end{theorem}

The starting point of the construction is the manifold $T^2\times \Sigma_g$, the product of a 2-torus and a surface of genus g, equipped with the product symplectic form. We gather the topological properties of this manifold that will be used in the following result. 

\begin{lemma} {\label{Lemma 23}}The first homology group $H_1(T^2\times \Sigma_g; \mathbb{Z}) \cong \mathbb{Z}^{2g + 2}$ is generated by $x, y, a_1, b_1, \cdots, a_g, b_g$. The second homology $H_2(T^2\times \Sigma_g; \mathbb{Z}) \cong \mathbb{Z}^{4g + 2}$ is generated by $T^2\times \{s_g\} (\{s_g\} \in \Sigma_g), \{t\}\times \Sigma_g (\{t\}\in T^2)$, and the tori $x \times a_i, y\times a_i, x\times b_i, y\times b_i$ where $i = 1, 2, \cdots, g$. The fundamental group $\pi_1(T^2\times \Sigma_g - \bigcup_{i = 1} ^{2g} T_i)$ is generated by the elements $x, y, a_1, b_1, \cdots, a_g, b_g$, and the relations 
\begin{center}
$[x, a_i] = [y, a_i] = [y, b_i a_i b_i^{-1}] = 1$ for $i = 1, \cdots, g$, and $[x, y] = 1 = [a_1, b_1[a_2, b_2] \cdots [a_g, b_g]$
\end{center}
hold in this group. The tori with the corresponding meridian, and Lagrangian push offs are given by
\begin{itemize}
\item $T_1: m_1 = x, l_1 = a_1, \mu_1= [b_1^{-1}, y^{-1}]$,
\item $T_2: m_2 = y, l_2 = b_1a_1b^{-1}, \mu_2 = [x^{-1}, b_1]$,
\item $T_3: m_3 = x, l_3 = a_2, \mu_3 = [b_2^{-1}, y^{-1}]$,
\item $T_4: m_4 = y, l_4 = b_2a_2b_2^{-1}, \mu_4 = [x^{-1}, b_2]$,
\item $T_5: m_5 = x, l_5 = a_3, \mu_5 = [b_3^{-1}, y^{-1}]$,\\
\vdots
\item $T_{2g - 1}: m_{2g - 1} = x, l_{2g - 1} = a_{g - 1}, \mu_{2g - 1} = [b_g^{-1}, y^{-1}]$, and
\item $T_{2g}: m_g = y, l_g = b_g a_g b_g^{-1}, \mu_{2g} = [x^{-1}, b_g]$.\\
\end{itemize}
\end{lemma}

The proof of this lemma is omitted. It is an straight-forward generalization of  \cite[Proposition 7]{[BK3]}; compare with Proposition \ref{Proposition 11}. The proof for both instances of the theorem consists of building a generalized complex 4-manifold with trivial Euler characteristic, trivial signature, and with the desired fundamental group. Then one concludes by blowing up $k$ points on such a manifold (Theorem \ref{Theorem 7}). Let us start with the case of surface groups.

\begin{proof} (Theorem 22). Perturb the product symplectic form of $T^2\times \Sigma_g$ \cite[Lemma 1.6]{[Go]} so that $T_1$ and $T_2$ become symplectic. Denote by $X_{S, g}(2)$ the twisted generalized complex 4-manifold obtained from $T_2\times \Sigma_g$ by applying $(1, 0, 0)$-surgeries on $T_1$ and on $T_2$. Theorem \ref{Theorem 6} implies that $X_{S, g}(2)$ admits a twisted generalized complex structure with two type change loci. We have $e(X_{S, g}(2)) = 0 = \sigma(X_{S, g}(2))$. Indeed, the Euler characteristic and the signature remain invariant under torus surgeries, so $e(X_{S, g}(2)) = e(T^2\times \Sigma_g) = 0$ and $\sigma(X_{S, g}(2)) = e(T^2\times \Sigma_g) = 0$. 
Moreover, the surgeries set $x = 1 = y$. Thus,
\begin{center}
$\pi_1(X_{S, g}(2)) = \pi_1(\Sigma_g)$.
\end{center}

The manifold $X_{S, g, k}(2)$ is a twisted generalized complex 4-manifold obtained by blowing up $X_{S, g}(2)$ at $k$ non-degenerate points (Theorems \ref{Theorem 6} and \ref{Theorem 7}). A straight-forward computation yields $e(X_{S, g, k}(2)) = k, \sigma(X_{S, g, k}(2)) = -k$. Since blow ups do not alter the fundamental group, we have $\pi_1(X_{S, g, k}(2)) = \pi_1(\Sigma_g)$. The twisted generalized complex structure with one type change locus is built by doing $(1, 0, 1)$-surgery on the Lagrangian $T_1$, and a $(1, 0, 0)$-surgery on the symplectic $T_2$.

Let us consider the case of free fundamental groups of rank $g$. We claim that there exists a  twisted generalized complex 4-manifold $X_{F, g}(i)$ with $i$ type change loci, with trivial Euler characteristic, trivial signature, and whose fundamental group is generated by $b_1, b_2, \ldots, b_g$ so that
\begin{center}
$\pi_1(X_{F, g}(i)) = \Z b_1 \ast \cdots \ast \Z b_g$.
\end{center}

The manifold $X_{F, g, k}(i)$ is obtained by blowing up $X_{F, g}(i)$ at $k$ non-degenerate points  (Theorem \ref{Theorem 7}). As in the previous case, by a straight-forward computation, one checks $e(X_{F, g, k}(i)) = k, \sigma(X_{F, g, k}(i)) = -k$, and
\begin{center}
$\pi_1(X_{F, g, k}) =  \overbrace{\Z \ast \cdots \ast \Z}^g$.
\end{center}

The twisted generalized complex manifold $X_{F, g}(i)$ is constructed as follows. We work out the case of one single type change locus (i = 1); the other cases follow from a small variation of the argument as in the surface groups case that was discussed before.
Perturb the symplectic form on $T^2\times \Sigma_g$ so that the torus $T_2$ becomes symplectic. Perform $(0, 1, -1)$-surgeries on $T_1$, $(1, 0, -1)$-surgery on $T_3$, and $(0, 1, -1)$-surgeries on $T_j$ for $4\leq j \leq g + 2$. Peform $(1, 0, 0)$-surgery on $T_2$.
The relation induced by the last surgery is $y = 1$; by using this on the relations introduced by the Luttinger surgeries, one sees $x = a_1 = a_2 =\cdots = a_g = 1$. This implies
\begin{center}
$\pi_1(X_{F, g}(1)) = \mathbb{Z} b_1 \ast \cdots \ast \mathbb{Z} b_g$.
\end{center}
By Theorem \ref{Theorem 6}, $X_{F, g}(1)$ is a twisted generalized complex manifold, and it has trivial Euler characteristic and signature zero.

The triviality of the Seiberg-Witten invariants of any manifold constructed above follows from Theorem \ref{Theorem 9}, by keeping track of the submanifolds on Lemma \ref{Lemma 23} and observing the existence of an imbedded 2-sphere of self-intersection zero.
\end{proof}

\section{On further research}\label{Section 7}

We finish the paper with two interesting questions.


\begin{question}\label{Question 1} What are the sufficient conditions for a manifold to admit a generalized complex structure in dimension four?
\end{question}

\begin{question}\label{Question 2} What is the relation between an almost-complex structure and a generalized complex structure on a given smooth 4-manifold?
\end{question}

\section{acknowledgments}

The author is indebted to Gil Cavalcanti, to Marco Gualtieri and to Nigel Hitchin for fruitful conversations that led to the writing of the present paper. Special thanks are due to Jonathan Yazinski for useful comments on an earlier draft. We thank Ronald Fintushel, Robert E. Gompf, Paul Kirk, Jonathan Normand, and Ronald J. Stern for helpful e-mail exchanges, useful discussions, and/or for their patience answering questions. Their expertise was of significant value in the production of the manuscript. We thank Utrecht University and University of Toronto for their hospitality during the production of the paper.\\ 

The Simons Foundation is gratefully acknowledged for the support under which this work was carried out.

\end{document}